\documentclass{amsart}  
 \usepackage{hyperref}   
 \hypersetup{
colorlinks = true,
} 
     
\usepackage{amsmath, amssymb, mathrsfs, amsfonts,  amsthm} 
\usepackage [all]{xy}

\DeclareMathOperator{\Ad}{Ad}

\DeclareMathOperator{\area}{area}
\DeclareMathOperator{\ind}{index}

\newtheorem {theorem} {Theorem} [section]

\newtheorem {remark/question} [theorem] {Remark/Question}
 
\newtheorem {conjecture} [theorem] {Conjecture}
\newtheorem{lemma}[theorem] {Lemma}

\newtheorem {proposition}  [theorem]{Proposition}

\numberwithin {equation} {section}

\begin{document}
\author {Yasha Savelyev}
\title {Yang-Mills theory and jumping curves}
\dedicatory {To Anna}
 \address{Calle Nicolas Cabrera 13 N 13-15 ICMAT, 28049 Madrid}
\email{yasha.savelyev@gmail.com}  
\begin{abstract} We study a smooth analogue of jumping curves of a holomorphic
vector bundle, and use Yang-Mills theory over $ S ^{2} $ to show that
any non-trivial, smooth Hermitian vector bundle $E $ over a smooth
simply connected manifold, must have such curves. This is used to give new
examples complex manifolds for which a non-trivial holomorphic vector bundle must have jumping curves in the classical sense, (when $c_1 (E)$
is zero). We also use this to give a new proof a theorem of Gromov on the norm of
curvature of unitary connections, and make the theorem slightly sharper.  Lastly
we define a sequence of new non-trivial integer invariants of smooth manifolds,
connected to this theory of smooth jumping curves, and make some computations of
these invariants.  Our
methods include an application of the recently developed Morse-Bott chain complex for the
Yang-Mills functional over $S ^{2} $.
\end{abstract}
\maketitle
\section{Introduction}
To study a holomorphic vector bundle $E$ over a complex manifold $X$,
one strategy is to extract information about $E$ from information obtained by
restrictions of $E $ to (families) of complex genus 0 curves in $X $. By
Birkhoff-Grothendieck theorem, a  complex rank $r$ holomorphic vector bundles over $ \mathbb{CP} ^{1} $ is
holomorphically isomorphic to $$\sum _{1 \leq i \leq r} \mathcal {O} (d_i),$$
with weights $d_i $ uniquely determined, (up to order) which makes their theory
particularly transparent. We shall say that a holomorphic
curve $u: \mathbb{CP} ^{1} \to X$ is a \emph{jumping curve} if the pullback
of $E$ to $ \mathbb{CP} ^{1}$ by $u$ is holomorphically non-trivial. When $c_1
(E) \neq 0 $ the classical definition is different but for our purpose we stick to the above
for convenience. It is not actually very hard to extend our arguments to the
case of classical jumping curves (when $c_1 (E) \neq 0$), but at the cost of
 much added complexity.
  The strategy then consists of studying splitting types of jumping curves. 
For $X = \mathbb{CP} ^{n}$, one may
take this approach very far, see for example \cite{bundles}, in particular the
authors prove that a  holomorphic vector bundle over $ \mathbb{CP}
^{n} $ is non trivial if and only if it has jumping lines. 

Jumping curves can be defined for a smooth
Hermitian vector bundle with a connection, as a direct generalization of the
holomorphic case. Let $X$ be a smooth manifold, $E$ a Hermitian vector bundle
over $X$, and $A$ a Hermitian, i.e. unitary connection on $E $. Given a smooth map $u: S ^{2} \to X$, the pull-back bundle $u ^{*} E$, with its pull back connection $u ^{*} A $,
has a natural induced holomorphic structure. It is the holomorphic structure
whose Dolbeault operator is induced by the anti-complex linear part of the
covariant derivative operator corresponding to $A $, this operator will be
denoted by $ \overline{\partial} _{A} $. This Dolbeault operator is necessarily
integrable, see for instance \cite{Yang-Mills}.  We shall say that $u$ is a
\emph{smooth jumping curve} if the $\overline{\partial} _{u ^{*}A} $
holomorphic structure on $u ^{*}E$ is non-trivial. We will also just say jumping curves where there is no possibility of confusion with holomorphic case. 

\begin{theorem} \label{thm.main} Let $X$ be a simply connected smooth manifold,
$E$ a   non-trivial, smooth Hermitian vector bundle over $X$, 
then for any Hermitian connection $A$ on $E$, there are smooth jumping
curves in any class $[u] \in \pi_2 (X)$. 
\end{theorem}
This is proved by studying the map $\Omega ^{2} X \to \Omega ^{2}BU (r) $
induced by the classifying map for $E$, via Yang-Mills flow, where $\Omega
^{2} X$ denotes the spherical mapping space. 
More precisely we use the energy flow for the
Yang-Mills functional, and at some points we use its Morse-Bott complex, which
for the base $S ^{2} $ has recently been nicely developed,
\cite{MorseYangMills}. 

Let us explain the relationship of the above to the holomorphic setting. Suppose
that $E \to X$ is a holomorphic vector bundle. Pick a Hermitian
structure $h$ on $E$. Then there is a unique compatible Hermitian connection $A
_{h}$ on $E$, inducing the given holomorphic structure, \cite{GriffithsHarris}.
From this it readily follows that a holomorphic curve $u: \mathbb{CP} ^{1} \to
X$, is a jumping curve if and only if it is a smooth jumping curve with respect to $A
_{h}$.  Let $ Hol^{ [u]} ( \mathbb{CP} ^{1}, X)  $ denote the space of based
holomorphic maps in class $[u] $. 
Then an  argument for theorem above, together with the discussion above give
the following. 
\begin{theorem} \label{thm.holomorphic} Suppose that $E$ is a
   non-trivial (topologically as
a complex vector bundle) holomorphic vector bundle over $X$, with $\pi_1 (X) =0$
and suppose that for every $k >0 $ there is a class $ [u_k] $ such that the inclusion 
$Hol ^{ [u_k]} ( \mathbb{CP} ^{1}, X) \to \Omega ^{2, [u_k]} X $
induces an isomorphism of homotopy groups, in range $[0,k]$. Then $E$ has 
(holomorphic) jumping curves. In particular this holds for $X = \mathbb{CP} ^{n}
$, a  generalized flag manifold, a toric manifold, and $X= \Omega G $, with $G $
compact Lie group.
\end{theorem}
Note that we can no longer control the class of jumping curves without stronger
assumptions. The hypothesis are known to be satisfied for $X= \mathbb{CP} ^{n} $
\cite{Segal.top.spaces.rat},
  flag manifolds for example \cite{guest.flag}, \cite{Kirwan}
toric manifolds \cite{guest.toric}, loop groups $\Omega G $, with $G $ compact
Lie group, for example \cite{Gravesen.loopgroup}.
Thus \ref{thm.holomorphic} also reproduces a slightly weakened version of the
result in \cite{bundles} mentioned above, weaker because in principle we obtain
only jumping curves rather than lines, and we must ask for topological
non-triviality rather than holomorphic. For the other cases the above existence
result on jumping curves seems to be new.
\subsection {Some new integer invariants of smooth vector bundles and smooth
manifolds} For a smooth complex rank $r$ vector bundle $E $ over $X$, and a
unitary connection $A$, if $u: S ^{2} \to X $ is a smooth  curve define
\begin{align} \label{eq.energy} |u| _{A}  &= \sum _{i} |d _{i}| ^{2}, \\
|u| _{A, \infty} & = \max _{i} |d _{i}|,
\end{align} 
 with $d _{i} $ the weights in the
Birkhoff-Grothendieck splitting of $ (u ^{*} E, \overline{\partial} _{u ^{*}A} )
$. Note that $ |u| _{A} $ is the energy of the homomorphism $S ^{1} \to U (n) $,
with weights $d _{i} $, after appropriate choice of normalization of 
 the bi-invariant metric on $U (n) $. This will come into the proof of
 \ref{thm.instability}.

Define: $$\zeta _{YM} (E, n) =  \inf _{A, [f], f' \in [f]} sup _{s \in S
^{n}} |f' (s)| _{A} \in \mathbb{N},$$ where the infinum is over all 
 $ [f] \in \pi_n (\Omega ^{2} X) $,  
 s.t. the map $\Omega ^{2} X \to \Omega ^{2} 
 BU (r)$ induced by the classifying map of $E $ is non-vanishing on $ [f] $, if
 there is no such $[f] $ set $ \zeta _{YM} (E, n)  =0$. Here $\Omega ^{2} X $
 denotes the component of the spherical mapping space in the constant
 class, (which is why the invariant is $ \mathbb{N} $ valued). Also define:
\begin{align*} 
\zeta _{YM} (E) &=  \inf _{n} \zeta _{YM} (E, n) \in \mathbb{N}, \\
 K _{YM} (X, r) &= \inf _{E} \zeta _{YM} (E) \in \mathbb{N},  \\
K _{YM} (X) & =\inf _{r>0} K _{YM} (X, r) \in \mathbb{N}. 
\end{align*}
 where the second infinum is over all rank $r>0$ complex vector bundles $E $
 with some non-vanishing Chern number. 
 \begin{theorem} \label{thm.nonzero} If $E$ is a non-trivial complex vector bundle over a simply connected $X$,
  then the sequence of positive integers $ \{\zeta _{YM} (E, n)\} $ cannot be
  identically $0$ and hence $ \zeta _{YM} (E) \neq 0 $,  $K _{YM} (X,r) \neq 0$,
  and $ K _{YM} (X) \neq 0 $.
\end{theorem}
\begin{conjecture} \label{prop.invariance} For $f: X \to Y$ a smooth homotopy
equivalence, and $E$ a complex vector bundle over $Y$:
\begin{equation*} \zeta _{YM} (E) = \zeta _{YM} (f ^{*} E), 
\end{equation*}
in particular the invariants $K _{YM} (X, r)$, $  K _{YM} (X)  $ of a
smooth manifold $X$ are homotopy invariants. 
\end{conjecture}
This conjecture is extremely likely. It would
follow immediately if the perturbation property needed in the proof of Theorem
\ref{thm.instability}, held for every $X $ as opposed to just $X =S ^{m} $. With
the perturbation property, it also becomes much easier to compute (at least in
principle) the above invariants.

We note that unlike Pontryagin and Chern numbers, the integers $\zeta _{YM}
(E,n) $ are not stable. That is, for a direct sum $E' =E \oplus  \epsilon ^{k}
$ of vector bundles with $\epsilon ^{k} $ the trivial $ \mathbb{C} ^{k} $ vector
bundle, we do not in general have $ \zeta _{YM}
(E,n) = \zeta _{YM} (E', n) $. In fact we have:
\begin{theorem} \label{thm.instability} For any  rank $r$ complex
vector bundle $E \to S ^{m}$,   and any $k   >0$, so that $n \leq
2 (r+k) -2 $ we have $$ \zeta _{YM} ( E \oplus  \epsilon ^{k}, m-2)  \leq
r+k.$$ On the other hand for $l=2$, or any $l $ of the form $l = \sum _{1
\leq i \leq r} |d_i| ^{2}$, with $d _{i} \in \mathbb{Z}$, $\sum _{i} d _{i} =0 $, 
  there is a complex rank $r$ vector bundle $E $ over $ S ^{m _{l}} $, with $
  \zeta _{YM} (E, m _{l} -2) = l$. 
 \end{theorem}
\subsection {Gromov norm}
Given a unitary connection $A$ on a rank $r$ complex vector bundle $E \to X $,
 a metric $g$ on $X$, and a fixed bi-invariant metric on $ \mathfrak {u} (r) $,
 we can define the norm of its curvature $R _{A}$ 
  by \begin{equation*} |R _{A}| _{g} \equiv \sup _{|v \wedge w| _{g}
 =1 } |R _{A} (v,w)|,
\end{equation*}
where $R _{A} $ is considered as a $ \mathfrak {u} (r) $
 valued $2$-form.
Then \ref{thm.main} gives another very different in nature proof of the
following theorem of Gromov. \begin{theorem} [\cite{Gromov}] For a simply
connected Riemannian manifold $X,g$, and $E$ a non-trivial Hermitian vector bundle over $X $, we have
\begin{equation*} \inf _{A} |R _{A}| _{g} >0. 
\end{equation*}
\end{theorem}
We can sharpen it as follows.
\begin{theorem}  \label{thm.lower.boundkarea} For $X,E,g$ as above:
\begin{equation*} \inf _{A} |R _{A}| _{g} \geq \inf _{A} \sup _{n, [f], f' \in
[f], s \in S ^{n}} \frac{ |f' (s)| _{A, \infty}} {\area _{g} (f'
(s)),} > 0,
\end{equation*}
with the supremum over all $ [f] \in \pi _{n} (\Omega ^{2} X) $, s.t. the map 
$\Omega ^{2} X \to \Omega ^{2} BU (r) $ is non-vanishing on $ [f] $, and 
with $\area _{g}$ denoting the area functional with respect to $g$.
\end{theorem} 
This is proved using Chern-Weil theory, in connection with our approach to
jumping curves via Yang-Mills theory.
\section {Acknowledgements} I am grateful to Fran Presas for related
discussions, Dietmar Salamon for providing some references, and  Egor
Shelukhin for supportive comments. This work was partially completed during my stay at CRM Barcelona and ICMAT Madrid, 
under funding of Viktor Ginzburg laboratory and Severo Ochoa project, I am very grateful for the support.
\section{Preliminaries} 
We briefly collect here some principal points of Yang-Mills theory, and set up
our conventions and notation. Our notation will mostly follow
\cite{yangmillsDavis} and \cite{Yang-Mills}.

The Yang-Mills functional denoted by $YM $ is defined on the space $ \mathcal
{A} $ of $G$ connections on a principal $G$ bundle $P $ over a Riemann
manifold $(X,g)$ by:
\begin{equation*} YM (A) =  \int _{X} ||F
_{A}|| ^{2} \,dvol _{g} , 
\end{equation*}   
for $A $ a $G $ connection, and $F _{A} $ the curvature 2-form on $X$, with
values in $\Ad P = P \times _{\Ad G} \mathfrak {g} $. The norm $
||\cdot|| $ above is the natural norm induced by the metric $ g $ on $X$ 
and a fixed bi-invariant inner product on $ \mathfrak {g} $, inducing a metric on $ P \times _{\Ad G} \mathfrak {g}  $.

In this note $G $ will be $U (r)$, sometimes $SU
(r) $, and $X = S ^{2} $. 
The value $YM (A)$, will be called \emph{Yang-Mills energy} of a
connection and critical points will be called \emph{Yang-Mills connections}. For
$X= S ^{2}$ or more generally a Riemann surface, the negative gradient flow for
$YM$ exists for all time on the space of smooth connections, and
 moreover gradient lines converge to Yang-Mills connections. 
We shall call
the negative gradient flow for $YM $ on $ \mathcal {A} $, with respect to the
natural $L ^{2} $ metric, the \emph{Yang-Mills flow}. The functional $YM$ and
the metric are equivariant with respect to the action of the gauge group of $P
$, i.e. the group of its $G$ bundle automorphisms, and the flow is defined on the
space of gauge equivalence classes. A crucial property of the 
Yang-Mills flow is that it preserves the  orbit of $A$ under the complexified gauge
group. In our case $ P $ is associated to some
complex vector bundle $E $, and this means that Yang-Mills flow preserves
the holomorphic isomorphism type of $(E, \overline {\partial} _{A})$.

For the proof of \ref{thm.instability} we also need some kind of Morse
homology for $YM$. Much additional work is required to set up Morse-Smale-Witten complex
on the space of gauge equivalence classes of connections. For $X=S ^{2} $ this is
started in \cite{yangmillsDavis}. Note however that unlike
\cite{yangmillsDavis} it is critical for us to work with the reduced gauge group which we call $ \mathcal
{G}_0 (P)$ and which consists of bundle automorphisms fixing the fiber over the
base point $0 \in S ^{2} $, as this group acts freely. When we say \emph{gauge
equivalence class}, we mean the equivalence class under the action of $ \mathcal
{G}_0$, as opposed to the full gauge group. In this case $YM $ on $
\mathcal {A}/ \mathcal {G}_0$ only gives rise to a Morse-Bott complex, full
details of this appear in \cite{MorseYangMills}.

\section {Proofs}
We need  a few preliminary steps. 
Fix a rank $r$ Hermitian vector bundle $E$ over $ \mathbb{CP} ^{1}$. 
Let $ \mathcal {G}_0 (E)$ denote the group of bundle automorphisms of 
$E$, fixing the fiber over $0$.  
Denote by $ \mathcal { \mathcal {A}} (E)$ the space of unitary connection on $E
$. It is easy to check that the group $ \mathcal {G}_0 (E)$ acts freely on $ \mathcal {A}
(E) $. Elements of $ \mathcal {A}
(E) / \mathcal {G}_0 (E)$ will be called \emph{gauge equivalence classes} of
connections. Let also $ \mathcal {E} $ denote the universal rank $r$ Hermitian
vector bundle over $BU (r) $, and let $\Omega ^{2, [u]}  BU (r)$ denote the smooth
spherical mapping space in the component of a map  $u: S ^{2} \to BU (r)$. 
To make sense of smooth here as well as of smooth connections on the universal
bundle to be used later, we only need to observe that $BU (r) $ has a
homotopy model that is a direct limit of smooth manifolds, $BU (r) \simeq 
\lim _{n} Gr _{ \mathbb{C}} (r, \mathbb{C} ^{n})$, with maps in the directed
system also smooth. This will be sufficient to make sense of what follows. 
For $v \in \Omega ^{2, [u]}  BU (r) $ set $v ^{*} \mathcal {E}  = \mathcal {E} 
_{v}$. 

\begin{lemma} Given a smooth unitary connection $A$ on the universal
bundle $ \mathcal {E}$ over $BU (r) $, there is a natural
induced map $W _{A, u}: \Omega ^{2, [u]}  BU (r) \to \mathcal {A} (
\mathcal {E} _{u})/ \mathcal {G}_0 (
\mathcal {E} _{u})$. 
\end{lemma}
\begin{proof} For $v \in \Omega ^{2, [u]}  BU (r) $ the bundles
   $\mathcal{E}_{u} $ and $\mathcal{E}_{v} $ are isomorphic via a bundle map
   fixing the fiber over $0$. We may get such an isomorphism by fixing a path
   $m$ from $u$ to $v$ and fixing an auxiliary connection on the induced bundle
   $\mathcal{E} _{m} \to \mathbb{CP} ^{1} \times [0,1]  $, trivial over $\{0\}
   \times [0,1]$. 
   
   So we may  define
   $$W _{A, u} (v) \equiv iso ^{*} A
_{v} \in \mathcal {A} (
\mathcal {E} _{u})/ \mathcal {G}_0 (
\mathcal {E} _{u}),$$ where $A _{v}$ is the connection  $v ^{*} A$ on $ \mathcal
{E} ^{r} _{v} $, for $iso$ any fixed isomorphism as above.
%
We need to check that this is well defined, i.e. independent of the choice of
$iso
$. But this is immediate for if $iso'$ is another such isomorphism, the
connections $ iso ^{*} A
_{v} $, $ {iso'} ^{*}A $ are obviously gauge equivalent.
\end{proof}
\begin{proposition} $W _{A,u}$ is a homotopy equivalence for any $A$.
\end {proposition}
\begin{proof}  The spaces $ \Omega ^{2, [u]}  BU (r) $ and, 
$ \mathcal {A} (
\mathcal {E} _{u})/ \mathcal {G}_0 (
\mathcal {E} _{u}) $ are  homotopy equivalent and have the homotopy
types of CW complexes, \cite[Proposition
2.4]{Yang-Mills},  and we show that
that our particular map $W _{A,u} $ is an isomorphism  on homotopy groups. We
show injectivity. Thus, suppose that $f'_1
= W _{A,u} \circ f_1$,  $f'_2=W _{A, u} \circ f _{2} $, for $f_i: S ^{k} \to  
\Omega ^{2, [u]}  BU (r)$, are homotopic. 
As $ \mathcal {A} (
\mathcal {E} _{u})/ \mathcal {G}_0 (
\mathcal {E} _{u}) \simeq B \mathcal {G}_0 (
\mathcal {E} _{u}) $, to a given $f: S ^{k} \to \mathcal {A} (
\mathcal {E} _{u})/ \mathcal {G}_0 (
\mathcal {E} _{u}) $  we have
associated principal $ \mathcal {G}_0 (
\mathcal {E} _{u}) $-bundles $ \mathcal {P} _{f} $. Let $ \mathcal {P} _{i} $
denote the principal $ \mathcal {G}_0 (
\mathcal {E} _{u}) $ bundles associated to $f _{i}' $, and let $pr: P _{i} \to S
^{k} $ denote the associated $ \mathcal {E}  _{u} $ bundles, (a bundle with
fiber modeled on the complex vector bundle $E _{u} $ over $ \mathbb{CP} ^{1} $). 
We may also think of $P _{i} $ as total spaces of  rank $r$ Hermitian
vector bundles over $S ^{k} \times S ^{2} $. For distinction let us call
them $E _{i}$. It is elementary to verify that $E _{i}$ are
classified by the maps $ cl_i: S ^{k} \times S ^{2} \to BU (r) $,
induced by $f_i$. If  $H: S ^{k} \times [0,1] \to  \mathcal {A} (
\mathcal {E} _{u})/ \mathcal {G}_0 (
\mathcal {E} _{u})  $ is a homotopy between $f'_1$, $f'_2 $, we get again an
induced  vector bundle $E _{H}$ over  $S ^{k} \times S ^{2}
\times [0,1]$. By observation above it's classifying map to $BU (r) $ induces a
homotopy between $ \widetilde{f}_i $, and consequently a homotopy between $f_i$. 

 Note that the
 homotopy class of $W _{A,u}$ is clearly independent of  $A$, consequently to
 show surjectivity it suffices to construct for each smooth map $f: S
^{k} \to \mathcal {A} (\mathcal {E} _{u})/ \mathcal {G}_0 (
 \mathcal {E} _{u}) $, a smooth connection $A _{f}$ on $ \mathcal {E}$ and $
 \widetilde{f}: S ^{k} \to \Omega ^{2, [u]}  BU (r) $, s.t.   $ W _{A _{f},u}
 \circ \widetilde{f}$ is homotopy equivalent to $f$. 
 Let $pr: P _{f} \to S ^{k} $
 denote the associated $ \mathcal {E}  _{u} $ bundle, constructed as before (a bundle with fiber modeled on the
 Hermitian vector bundle $E _{u} $ over $ \mathbb{CP} ^{1} $). We now construct a
 smooth family  of unitary connections $ \{ {A} _{s}\} $,
 $s \in S ^{k} $, with $A _{s} $ a connection on the fiber $pr ^{-1} (s)$, with
 the property that $A _{s} = f (s) \in \mathcal {A} (
 \mathcal {E} _{u})/ \mathcal {G}_0 (
 \mathcal {E} _{u}) $ for every $s $. 
 Smooth here is in the sense that the
 family
 $ \{A _{s}\} $ is smooth in smooth trivializations of $P _{f} $.
 Cover $S ^{k}
 $ by a pair of disks 
 $ \{D ^{k}  _{i}\} $, with intersection diffeomorphic to $S ^{k-1} $
 and pick any pair of smooth sections $\{A _{s} ^{ì}  \}$ over $D ^{k}  _{i} $ of $f
 ^{*} (\mathcal{A}(E _{u} ))$. This sections may not coincide but they are related
 by a  transition map $$g: U  _{1} \cap U _{2}   \simeq S ^{k-1}  \to \mathcal{G}_{0}(\mathcal{E}_{u} ).  $$
 In particular they induce as required a well defined smooth family $\{A _{s} \}$ on $$P _{f} \simeq \mathcal{E}_{u}
 \times D ^{k} \sqcup _{g} \mathcal{E}_{u} \times D ^{k},    $$
with the latter being the bundle obtained by using $g$ as a clutching map.
Let $E _{f} \to S ^{k} \times S ^{2}  $ be the associated 
Hermitian vector bundles. The family $\{A _{s} \}$ may be extended to a
unitary connection $A _{f} $ on $E _{f} $. Let $cl _{f}: S ^{k} \times S ^{2} \to BU
(r)   $ be an embedding classifying $E _{f} $. We may then take any connection
$A$ on the universal bundle so that its pullback by $cl _{f} $ is $A _{f} $,
and we take the map $\widetilde{f} $ to be the map induced by $cl _{f} $.
\end{proof}
\begin{proof} [Proof of \ref{thm.main}] Suppose that $E$ has complex rank $r$.
Clearly we may suppose that $c _{1} (E) =0$, so that $cl_E: X \to BU (r)$ vanishes
on $\pi_2$. As $E$ is non-trivial, the classifying map $cl_E: X \to BU (r)$, 
must be non-vanishing on some $\pi _{k}$, with $k>2$. Consequently, since $X $
is simply connected the induced map $ \widetilde{cl}_E: \Omega ^{2,
[const]} X \to
\Omega ^{2, jconst]} BU
(r) $, is non-vanishing on some $[f] \in \pi _{k} (\Omega ^{2, [const]} X )$ with $k>0 $.
 Approximate the classifying
map by a smooth embedding $cl_E $, and push-forward the connection $A $ to a
connection on the universal bundle $ \mathcal {E} $ over the image of $f _{E} $, then extend to a connection $ \widetilde{A}$ on $ \mathcal {E} $ over  $BU (r)$.
 (There are no obstructions to existence of
such an extension, it can be constructed via partitions of unity for example).
Now suppose that there are no jumping curves in $X$ for $E, A$ in class $[u]$,
then we claim  that the map $\Omega ^{2, [u]} X
\to \Omega ^{2, [const]} BU (r) $ 
induced by the classifying map vanishes on all homotopy
groups, and consequently  $\Omega ^{2, [const]} X \to \Omega ^{2, [const]} BU (r) $
also vanishes on all homotopy groups, which is a contradiction. Indeed let $f: S ^{k} \to \Omega ^{2, [u]}X  $, $k>1$ be a continuous map.
We show that the map 
$$ \widetilde{f} = W _{ \widetilde{A}, [const]} \circ \widetilde{cl}_E \circ f:
S ^{k} \to \mathcal {A} (\epsilon ^{r}  )/ \mathcal {G}_0 ( \epsilon ^{r}  ),$$
is null-homotopic,
where $const $ denotes the constant based map $const: S ^{2}
\to BU (r)$, and $ \epsilon  ^{r} $ is shorthand for the trivial $
\mathbb{C} ^{r} $ vector bundle $ \mathcal {E} _{const}$. 
This will be a contradiction and
hence conclude the proof.

 Indeed let $ \phi _{YM} (t) $ denote the time $t $
flow map of the  Yang-Mills flow on $\mathcal {A} (\epsilon ^{r}  )/ \mathcal
{G}_0 ( \epsilon ^{r}  )$.  We need a
 pair of facts: the isomorphism class of the induced holomorphic
structure on $ \epsilon ^{r}  $ by elements $ | {A}| \in \mathcal {A} (\epsilon
^{r} )/ \mathcal {G}_0 ( \epsilon ^{r}  )$ is constant along negative gradient flow lines of the Yang-Mills functional,
which follows by \cite[8.12]{Yang-Mills}, and a critical $|{A}| $
for the Yang-Mills functional (a.k.a a Yang-Mills connection) determines a
non-trivial holomorphic structure on $ \epsilon ^{r}   $, so long as it is not the 
gauge equivalence class of a trivial connection. 
 Consequently, as the holomorphic structure induced by
each $\widetilde{f} (s), $ $s \in S ^{k} $ is trivial by assumption, the negative gradient lines $t
\mapsto \phi _{YM} (t)  (\widetilde{f} (s))$ must converge to the gauge
equivalence class of a trivial connection. Reparametrizing Yang-Mills
flow, (by Yang-Mills energy),  we obtain
a null-homotopy of $ \widetilde{f} $.
\end{proof}  
\begin{proof} [Proof of \ref{thm.nonzero}] This follows readily by the proof
of \ref{thm.main}.
\end{proof}
\begin{proof} [Proof of \ref{thm.holomorphic}] As previously observed the map
$ \widetilde{cl}_E: \Omega ^{2} X \to \Omega ^{2} BU (r) $ induced by the
classifying map for $E$, must be non-vanishing on some $ [f] \in \pi_k  (\Omega
^{2, [const]} X)$, with $k>0 $. Let $ [u_k] $ be as in the hypothesis and suppose that there are no class $ [u_k] $,
holomorphic jumping curves. Let $ \widetilde{f}: S ^{k} \to
\Omega ^{2, [u_k]} (X) $, be a representative for $iso \circ f $, for $iso:
\Omega ^{2, [const]} X \to \Omega ^{2, [u_k]} X $ the canonical homotopy equivalence,
so that $ \widetilde{f} $ is in the image of the inclusion $Hol ^{[u_k]} (
\mathbb{CP} ^{1}, X) \to \Omega ^{2, [u_k]} X $. Consequently there are no
smooth jumping curves in the image of $ \widetilde{f}$. Thus running the
argument in \ref{thm.main} we get that $\widetilde{cl}_E\circ \widetilde{f}$ is
null-homotopic, but then $ \widetilde{cl}_E $ must be vanishing on the class $ [f] $ as well, a contradiction.
\end{proof}
\begin{proof} [Proof of \ref{thm.lower.boundkarea}]
This follows by the proof of \ref{thm.main}  once we note: 
\begin{lemma} Let $E,X, A, g$ be as in the hypothesis. Suppose $u: \mathbb{CP}
^{1} \to X $ is a smooth jumping curve. Then the norm of the curvature of $u
^{*} E, u ^{*} A$, $|R ({u ^{*}A})| _{g} $ is at least $|u| _{A, \infty}/area
_{g} (u) $.
\end{lemma}

\begin{proof} Indeed for the holomorphic structure on $u ^{*} E $, induced by $u
^{*}A $, $u ^{*} A$  preserves the Grothendieck splitting of $u ^{*} E$.
Consequently the norm of the curvature of $u ^{*}A$ on the maximal weight
subspace $ \mathcal {O} _{max}$ of $ u ^{*} E $ cannot exceed the norm of
curvature of $u ^{*}A $ on $u ^{*} E$, see \cite[Page 79]{GriffithsHarris}. This
follows from ``positivity'' of curvature on a holomorphic vector bundle.

 As
$c_1 ( \mathcal {O} _{\max}) =|u| _{A, \infty} $, by assumption, Chern-Weil
theory gives a ready estimate 
\begin{equation*} |R(u ^{*} A| _{ \mathcal {O} _{\max}})| _{u ^{*} g}
\cdot \area _{g} (u) \geq |u| _{A, \infty}.
\end{equation*}
\end{proof} 
\end{proof}
\begin{proof} [Proof of \ref{thm.instability}] 
 As we are working in the component of $\Omega ^{2}X $ in the constant class we
 may  restrict our arguments to  $SU (r)$ connections on $E$, and $ \mathcal
 { \mathcal {A}} ( \epsilon ^{r})$ will denote the space of smooth $SU (r) $ connections on $ \epsilon ^{r}$.

 Morse
 theory for the Yang-Mills functional $YM$ on $ \mathcal {A} ^{r} \equiv
 \mathcal {A} ( \epsilon^r)/ \mathcal {G}_0 ( \epsilon^r) $ is known
 to be essentially equivalent to Morse theory on $\Omega SU (r) $ for the energy
 functional,  \cite{Gravesen},  \cite{YangMillsandEnergy}.  Although we
 run the argument with Yang-Mills functional, it will be helpful to refer to the correspondence in a
 few instances. Moreover the energy
 flow picture, as developed in great detail in \cite{Segal}, is probably more
 accessible, and may help the reader to get intuition. The correspondence is
 given via radial trivialization. We briefly sketch this. Given an element $ |A|
\in  \mathcal {A} ^{r}$ we get a loop  $Rad ( | {A}|) \in \Omega SU (r)
$, by fixing an $S ^{1}$-family of rays $ \{r _{\theta}\}$ from $0$ to $\infty \in
\mathbb{CP} ^{1}$, identifying the fibers over $0,\infty $ via $A $-parallel transport along $r_0
$, and then getting a loop in $SU (r) $ via $A$-parallel transport along
rays $\{r_ {\theta}\}$, for $A \in | {A}| $, which is well defined with
respect to  equivalence under action of $ \mathcal {G}_0 ( \epsilon ^{r} )$. 
The functions and their flow are then shown to behave well with respect to this correspondence. 
In particular critical points and their indices correspond. Thus radial trivialization of a
Yang-Mills connection is an $S ^{1}$ subgroup of $SU (r) $. 


We shall use the construction of Morse-Bott-Smale-Witten complex for $YM $, in
\cite{MorseYangMills}. The specific complex that we use is the so called
\emph{cascade complex} originally appearing in Fraunfelder
\cite{citeUrsArnoldGivental}, in a Floer
theory setting. For another discussion of the cascade complex also comparing various different
approaches to Morse-Bott complex, see \cite{cascades}.
Very
briefly given a Morse-Bott function $h$ on a manifold $M$, $g$
a metric on $M$, so that $(h,g)$ is a Morse-Bott-Smale pair and
 $\{aux _{i}\} $
auxiliary Morse-Smale (with respect to $g$) functions on the critical manifolds
$\{C _{i} \}$, one constructs a complex $C (h, \{aux _{i} \}, g)$ which is
generated by critical points of $\{aux _{i} \}$, graded by 
\begin{equation*}
\ind (p) = \ind _{MB}  (C _{i} ) + \ind _{aux_i} p,
\end{equation*}
for $p \in crit (C _{i}, aux _{i})  $, $\ind _{MB} $ the Morse-Bott index and
$\ind _{aux _{i} } p$ the Morse index with respect to $aux _{i} $. The
differential is obtained by count of $\mathbb{R}$ reparametrized,
Fredholm index 1, \emph{cascade flow lines} abutting  to
generators.  The (unbroken) cascade flow lines are
formal concatenations of negative gradient trajectories $$(\gamma _{aux _{i_1} },
\gamma _{h}, \gamma _{aux _{i_0} }),$$ where $\gamma _{h}: (-\infty,
\infty) \to M$  is a
negative gradient trajectory for $h$, $\gamma _{aux _{i_1} }: [0, \infty)
\to M  $ is a negative gradient
trajectory for $aux _{i_1} $, and likewise $\gamma _{aux _{i_0} }: (-\infty, 0] $
is a negative gradient trajectorie for $aux _{i_0} $.
It is shown in \cite{cascades} in finite dimensional context how to use the functions $aux _{i} $ to
perturb $h$ to a Morse-Smale function $h'$, so that there is a index
preserving
correspondence between the generators of the cascade complex and the
critical points of $h'$, and a certain correspondence between the
cascade flow lines and the classical flow-lines for $(h',g)$. Passing to our infinite dimensional setting in principle
requires a bit of care, but we leave this out as it is not very
interesting.  We shall say in this context that the flow of
$h$ is \emph{cascade perturbed}.

The critical sets of $YM$ on $ \mathcal {A} ^{r} $ split
into disjoint unions of submanifolds $\Lambda _{ \textbf{d}} $, composed of
(Yang-Mills) connections, inducing the holomorphic structure of type $
\textbf{d}$, where $ \textbf{d} $
by Birkhoff-Grothendieck can be taken to mean an unordered collection of $r$
integers. The manifolds $ \Lambda _{ \textbf{d}} $, which are isomorphic to the
manifolds of $S ^{1}$ subgroups of $SU (r)$, conjugate to the diagonal $S ^{1} $
subgroup with weights $ \textbf{d} $, are certain complex flag manifolds and
admit perfect Morse functions. Consequently, the homology of $ \mathcal {A} ^{r} $ is
computed by the cascade Morse-Bott-Smale-Witten complex,
 with  vanishing differentials, as the Morse-Bott index in this
case is always even, (see elaboration below). We will not mention the auxiliary Morse-Smale functions on
the critical manifolds explicitly, and just write  $ \mathcal {C}_* (YM)$, for
this complex.

Up to degree $2r-2$ all the generators of  $ \mathcal {C}_*
(YM)$, come from critical manifolds $\Lambda _{ \textbf{d}} $, with $ \textbf{d}
$ having all weights $1,-1$ or $ 0 $. This readily follows upon computing the Morse-Bott index along $ \Lambda _{ \textbf{d}} $, (i.e. the normal component of
Morse index). As the Morse index for $YM$ of a critical $
| {A}| \in \mathcal {A} ^{r}$, coincides with the Morse index of $ rad (
| {A}|) \in \Omega SU (r)$, we may use the classical calculation of Morse index
 of geodesics in $ SU (r) $, see for example \cite{Morse.theory}. This tells us that for a homomorphism $\gamma: S ^{1}  \to
SU (r)$, conjugate to
\begin{equation*} \gamma = \left ( 
\begin{array} {ccc} e ^{2 \pi i a_1}  & & \\
& e ^{2 \pi_i a_2} & \\
& & \ldots \\
\end{array} \right )
\end{equation*}
with weights $a _{1}, \ldots, a_r $, organized so that $a_1 \geq a_2 \geq \ldots
\geq a _{r} $ the Morse index is given by
\begin{equation*} \sum _{i < j} 2|a_i -a_j| -2.
\end{equation*}
It is then easy to verify that if $ \textbf{d} = \{a_1, \ldots, a_r\}$ does not
satisfy that all $a _{i} $ are either $1, -1$ or $0 $, then the Morse index is
strictly greater than $2r-2 $, see \cite[page 132]{Morse.theory} for a similar
calculation. 

We may then use unstable manifolds of the cascade pertubed Yang-Mills flow, to induce a cellular
decomposition of $\mathcal{A} ^{r} $, with all cells of even dimension.
Likewise there is a dual stratification of $ \mathcal {A} ^{r} $ by
finite co-dimension strata, intersecting the cascade perturbed unstable manifolds 
transversally, corresponding to cascade pertubed stable manifolds for $YM$.
 It follows that given a map $f: S ^{k} \to \mathcal {A} ^{r}$ with
$ k \leq 2r-2$, after a small perturbation of $f$ the (reparametrized by
 energy) (cascade perturbed) Yang-Mills flow will give a homotopy of $f$ into the $2r-2$
skeleton of $ \mathcal {A} ^{r}$. 
That is, we just have to take a perturbation of $f $ so that we 
are transverse to the finite codimension strata. This perturbation exists by
general differential topology considerations. Let's now apply this.

Following the proof of \ref{thm.main}, we have a map $$\Omega ^{2} S ^{m} \to
\mathcal {A} ^{r+k}$$
induced by the classifying map for $E \oplus \epsilon ^{k}$. Take $k$ so that
$n \leq 2 (r+k)-2$.  It follows by discussion above that the composition $$
\widetilde{f} = W _{ \widetilde{A}, [const]} \circ \widetilde{cl} _{E \oplus \epsilon ^{k}} \circ f: S ^{n} \to \mathcal {A} ^{r+k}  $$ 
can be perturbed so that the energy flow (after reparametrization) takes $
 \widetilde{f}$ into the $ 2 (r+k)-2$ skeleton of  $ \mathcal {A} ^{r +k}$.

 Moreover we have:
 \begin{lemma} \label{lemma.perturb} Any abstract perturbation of $
\widetilde{f} $, can be obtained via perturbation of the connection $A$ on $E $. 
\end{lemma}
 If we accept this for the moment then the first part of the theorem
 follows, as  the Birkhoff-Grothendieck splitting type of elements in the $
 2(r+k)-2$ skeleton, by preliminary discussion above, is such that all weights are either $1,-1 $ or $0 $. Consequently,  (after
 perturbation of $A$) $ |f (s)| _{A} \leq r + k$ for each $s \in S ^{n} $.
 \begin{proof} [Proof of Lemma \ref{lemma.perturb}]  Let us then verify the
 claim about the perturbation. This is where the assumption $X= S ^{n} $ comes into the proof.    Take a representative $f: S
 ^{n} \to \Omega ^{2} S ^{m} $, for $ [f] $ so  that the induced map $g: S ^{n} \times S ^{2} \to S ^{m}$ is an embedding
 outside  the submanifolds $S ^{n} \times \{0\}$, $ \{0\} \times S ^{2} $ which
 are collapsed to a point. 
Given an abstract perturbation $ \widetilde{f}'
 $ of $ \widetilde{f} $, with $ \widetilde{F}: S ^{n} \times I \to \mathcal {A}
 ^{r+k} $ a homotopy between $ \widetilde{f}$, and $ \widetilde{f}'$, we get an
 associated structure group $ \mathcal {G}_0 ( \epsilon ^{r+k}) $ bundle $$\pi
 _{ \widetilde{F}} :P _{ \widetilde{F}} \to S ^{n} \times [0,1]$$ with fiber $
 \mathbb{C} ^{r+k} \times S ^{2} $. We also get a bundle $G _{ \widetilde{F}}
 \to S ^{n} \times [0,1] $, with fiber over $ (s,t)$ the space of connections on $\pi _{ \widetilde{F}} ^{-1} (s,t)$, in the
 gauge equivalence class of $ \widetilde{F} (s,t) $. Note that this is in
 general not isomorphic to the  principal $ \mathcal {G}_0 ( \epsilon ^{r+k})
 $ bundle pulled back by $ \widetilde{F} $. 
 By construction $G _{
 \widetilde{F}} $ has a section over $S ^{n} \times \{0\} $, use homotopy
 lifting property to obtain a section $ \mathcal {S} _{ \widetilde{f}'} $ over $
 S ^{n} \times \{1\} $.  $ \mathcal {S} _{ \widetilde{f}'}$ can be thought of as smooth family
 of connections $ \{A _{s}\} $  over $\{ \{s\}
 \times S ^{2} \subset S ^{n} \times S ^{2} \}$, for the $ \mathbb{C} ^{r+k} $
 bundle $E _{ \widetilde{f}'} \to S ^{n} \times S ^{2}$ canonically associated
 to $P _{ \widetilde{F}} $. Note that $ E _{ \widetilde{f}'} $ is canonically
 trivial over  $S ^{n} \times
\{0\} $, $ \{0\} \times S ^{2}$, consequently we may extend $ \{A _{s}\} $ to a
connection ${A}''$ on $E _{ \widetilde{f}'} \to S ^{n} \times S ^{2}$, trivial
over $S ^{n} \times
\{0\} $, $ \{0\} \times S ^{2}$. (The extension may be obtained via classical
partition of unity argument.) It follows that we may push forward $A'' $ via $g
$ to a connection $A' $ on $E $, with exactly the right property.
\end{proof}

%

We now verify the second part of the theorem. We will say that a skeleton of $
\mathcal {A} ^{r} $ (in the previous sense) has energy $l$, if all the
generators of $ \mathcal {C}_* (YM) $ in it have $YM$ energy less then or equal
to $l$. As $ \mathcal {C}_* (YM) $ is perfect, skeleta with distinct energy have distinct homology.
  From this it follows that a skeleton $ \mathcal {S} $ with energy $l$ admits a
  non null-homotopic based map $f: S ^{m _{l}-2} \to \mathcal {S} $, which
  cannot be pushed below the energy level $l$. As otherwise $ \mathcal {S} $ could be
  compressed into a lower energy skeleton, which would contradict the 
  above. (The compression argument is just Whitehead's
 theorem, see for example \cite{Hatcher}). 
As the map $W _{A, const}$ is a homotopy equivalence for every $A$, we may find
a $g: S ^{m _{l}-2} \to \Omega ^{2} BU (r)$, s.t. $W _{{A, const} _{*}} [g] =
[f]$. Following the proof of Lemma \ref{lemma.perturb} we get a family $ \{A'
_{s}\} $ of connections on $\{E'_{f}| _{ \{s\} \times S ^{2}} \}$ where $E' _{f}
$ is the $ \mathbb{C} ^{r} $ vector bundle over $ S ^{m _{l}-2} \times S ^{2}
$ associated to $f$ as in the proof of \ref{lemma.perturb}. Extend $ \{A'_s\} $
to a connection $ A' _{f}$ on $E _{f} $, trivial over $S ^{m_l-2} \times \{0\}
$, where the bundle is canonically trivialized. By construction $A' _{f} $ is also
the trivial connection over $ \{0\} \times S ^{2} $. It follows that $E' _{f}$
induces a vector bundle over 
$$S ^{m_l} = S ^{m_l-2} \bigvee S
^{2},$$
and $A' _{f} $ induces a connection $A _{f} $ on $E _{f} $. By construction the infinum in $ \zeta _{YM} (E _{f}, m
_{l})$ is attained on $A _{f}$. 
%
%
%

The proof is finished as the normalization for the bi-invariant metric on $SU
(r) $ that we take, is such that an $S ^{1} $ geodesic with weights $ \{d_i\} $,
has energy given by \eqref{eq.energy}.
\end{proof}
\bibliographystyle{siam}   

\end{document}